\documentclass[11pt,a4paper]{article}
\usepackage{a4wide}
\usepackage{amssymb} %Box
\usepackage{amsmath} %Operator
\usepackage{amsthm} %proof
\usepackage{tikz}

\newtheorem{theorem}{Theorem}

\newtheorem{claim}{Claim}

\title{A precolouring extension of Vizing's theorem}

\author{Ant\'onio Gir\~ao\,\thanks{\,Department of Pure Mathematics and
    Mathematical Statistics, University of Cambridge, Cambridge, UK;
    \texttt{A.Girao@dpmms.cam.ac.uk}. This author would like to thank Radboud University for its hospitality.}
  \and Ross J. Kang\,\thanks{\,Department of Mathematics, Radboud
    University, Nijmegen, Netherlands; \texttt{ross.kang@gmail.com}. Supported by a Vidi grant (639.032.614) of the Netherlands Organisation for Scientific Research (NWO).}}

\begin{document}

\maketitle
\begin{abstract}
Fix a palette $\mathcal K$ of $\Delta+1$ colours, a graph with maximum degree $\Delta$, and a subset $M$ of the edge set with minimum distance between edges at least $9$. If the edges of $M$ are arbitrarily precoloured from $\mathcal K$, then there is guaranteed to be a proper edge-colouring using only colours from $\mathcal K$ that extends the precolouring on $M$ to the entire graph. This result is a first general precolouring extension form of Vizing's theorem, and it proves a conjecture of Albertson and Moore under a slightly stronger distance requirement. We also show that the condition on the distance can be lowered to $5$ when the graph contains no cycle of length $5$.
\end{abstract}

%MSC class: 05C15

\section{Introduction}

Suppose that we need to plan a sports competition for which some of the matches have already been prescheduled. Under what circumstances do we find that the prescheduled matches do not adversely affect the eventual length of the competition schedule? This problem may be modelled by edge-precolouring extension: under what conditions is it always possible, given a graph and some partial proper (pre)colouring of its edges, to extend the precolouring to an optimal proper edge-colouring of the entire graph?

Recall that the classical theorem on proper edge-colouring due to Vizing~\cite{Viz64} (cf.~also Gupta~\cite{Gup74}) states that the chromatic index $\chi'(G)$ of a simple graph $G$ is either the maximum degree $\Delta(G)$ of $G$ or one larger.

In fact, we are mainly interested in whether there is some constant $c$ such that, if $\Delta(G) +1$ colours are allowed and the precoloured edges have pairwise distance at least $c$, then any precolouring can be extended to a proper edge-colouring of all of $G$. To be clear, when we refer to the distance between two edges we mean the number of edges in a shortest path between two of their endpoints. 
%(regardless of whether $\chi'(G)$ is $\Delta(G)$ or $\Delta(G)+1$)

Albertson and Moore~\cite{AlMo01} conjectured that such a constant $c$ exists and equals $3$, noting that the first case $\Delta(G)=3$ follows from~\cite{JMS98}.
In earlier work together with Edwards, van den Heuvel, Puleo, and Sereni~\cite{EGHKPS16+}, we showed that $c\ge 2$ (if it exists) and that 
no such constant $c$ exists if we instead  only allow $\Delta(G)$ colours for graphs $G$ for which $\chi'(G)=\Delta(G)$.%~\cite{EGHKPS16+}.

Our goal in the present paper is to show that such a constant $c$ does indeed exist. In particular, any precolouring of a graph $G$ where the precoloured edges are at pairwise distance at least $9$ can be extended to a proper edge-colouring of $G$ using at most $\Delta(G)+1$ colours. 

We should point out that our result fits within the broader context of vertex-precolouring extensions, which has seen a great deal of activity, cf.~e.g.~\cite{AlMo99,Tuz97}. This is an important topic in chromatic graph theory, especially due to Thomassen's ingenious use of precolouring extension to prove that all planar graphs are $5$-choosable~\cite{Tho94}. Notably, in a strikingly short answer to a related question of Thomassen, Albertson~\cite{Alb98} showed that,
given a graph $G$ with chromatic number $\chi(G)=r$ and a set $P$ of vertices with pairwise distance at least $4$, any precolouring of $P$ from a palette of $r+1$ colours extends to a proper $(r+1)$-colouring of $G$. Moreover, the distance condition $4$ is best possible. %There has been a great deal of activity on precolouring extension, .

Most previous work on (vertex-)precolouring extension allows for one additional colour as in Albertson's seminal result. We highlight two relevant exceptions. Albertson and Moore~\cite{AlMo01} considered how to extend partial $r$-colourings of $r$-chromatic graphs, and proved several sharp results; however, these results only apply to graphs that possess a special $r$-colouring. Later, Axenovich~\cite{Axe03} and Albertson, Kostochka and West~\cite{AKW04} proved that $\Delta$-precolourings of a set of vertices with minimum distance $8$ can be extended to full $\Delta$-colourings for any graph of maximum degree at most $\Delta\ge 3$ apart from $K_{\Delta+1}$, thereby extending Brooks' theorem.

It is worth noting that an edge-precolouring paper of Marcotte and Seymour~\cite{MaSe90} is one of the first papers on precolouring extension, even if it considered the problem differently, from the viewpoint of polyhedral optimisation.
Note, moreover, that a particular case of edge-precolouring extension was already considered in the 1960s in relation to Evans' conjecture, cf.~e.g.~\cite{Sme81}, by the interpretation of proper edge-colourings of complete bipartite graphs in terms of Latin squares and rectangles.

Here is our main result, restated slightly more generally in terms of multigraphs.

\begin{theorem}\label{thm:main}
Let $G$ be a multigraph of maximum edge multiplicity $\mu$ and maximum degree $\Delta$, and let $M$ be a set of edges such that the minimum distance between two edges of $M$ is at least $9$.
If $M$ is arbitrarily precoloured from the palette $\mathcal K = [\Delta+\mu] = \{1,\dots,\Delta+\mu\}$, then there is a proper edge-colouring of $G$ using colours from $\mathcal K$ that agrees with the precolouring on $M$.
\end{theorem}

\noindent
Note that when $M$ is empty, this result is exactly the aforementioned classic theorem of Vizing and Gupta applied to the multigraph case.

To prove Theorem~\ref{thm:main} we make use of a result of Berge and Fournier~\cite{BeFo91}, which is an edge-precolouring extension result in the case when the precoloured set uses only one colour. We need only a special case. 

%[Professor Bollobas suggests to remove the reference of the next theorem]

\begin{theorem}[Fournier, cf.~\cite{BeFo91}]\label{thm:BeFo}
Let $G$ be a multigraph of maximum edge multiplicity $\mu$ and maximum degree $\Delta$,
and let $M$ be a maximal matching of $G$.
Then there exists a proper $(\Delta+\mu-1)$-edge-colouring of $G\setminus M$.
\end{theorem}

\noindent
Like the proof in~\cite{BeFo91}, our proof of Theorem~\ref{thm:main} is short. Roughly, we take the following strategy.
Let $M' \supseteq M$ be a maximal matching of $G$.
By Theorem~\ref{thm:BeFo}, there is a proper edge-colouring of $G\setminus M'$ using only colours from $[\Delta+\mu-1]$.
We would like to colour all edges of $M'\setminus M$ with colour $\Delta+\mu$. For every $1\leq i \leq \Delta+\mu-1$, if there is an edge of $G\setminus M'$ that is coloured $i$ and incident to some edge precoloured $i$, then we would like to recolour that edge with the colour $\Delta+\mu$.
In the proof, we use a recolouring argument, using Vizing fans, to help us resolve the problems that arise.

Using the same strategy, we also show how we can afford to relax the distance constraint on the precoloured matching provided we impose a mild structural constraint on the graph.

\begin{theorem}\label{thm:c5free}
Let $G$ be a multigraph of maximum edge multiplicity $\mu$ and maximum degree $\Delta$. Suppose $G$ contains no cycle of length $5$ as a subgraph. Let $M$ be a set of edges such that the minimum distance between any two edges is at least $5$. If $M$ is arbitrarily precoloured from the palette $\mathcal{K}=[\Delta+\mu]=\{1,\ldots, \Delta+ \mu\}$, then there is a proper edge-colouring of $G$ using colours from $\mathcal{K}$ that agrees with the precolouring on $M$. 
\end{theorem}

\section{Proofs}

\begin{proof}[Proof of Theorem~\ref{thm:main}]
Let $\Phi : M\to \mathcal K$ be a precolouring of $M$. For each $i\in \mathcal K$, write $M_i \subseteq M$ for the edges precoloured $i$.
Let $\alpha$ be the cardinality of a smallest matching $M' \supseteq M$ such that there exists a proper $(\Delta+\mu-1)$-edge-colouring of $G\setminus M'$.
By Theorem~\ref{thm:BeFo} $\alpha$ is well defined.

For a matching $M' \supseteq M$ and a proper edge-colouring $\varphi : E(G\setminus M') \to [\Delta+\mu-1]$, we define $\Delta+\mu-1$ sets $A_1^{M',\varphi},\dots,A_{\Delta+\mu-1}^{M',\varphi}$ as follows. Let $1\leq i \leq \Delta+\mu-1$. For each endpoint $u$ of an edge of $M_i$, we let $A_u^{M',\varphi}$ denote all edges that are contained in the maximal path $P_u^{M',\varphi}$ beginning at $u$ that alternates between edges coloured $i$ by $\varphi$ and edges of $M'\setminus M$.
We shall write $P_u^{M',\varphi}=w^u_0e^u_0w^u_1e^u_1w^u_2e^u_2\cdots$ with vertices $w^u_k$ and edges $e^u_k$, where $w^u_0=u$ and $\varphi(e^u_0)=i$.
We take $A_i^{M',\varphi}$ to be the union of all $A_u^{M',\varphi}$, with $u$ an endvertex of an edge of $M_i$. So $A_i^{M',\varphi}\cup M_i$ induces a disjoint union of paths and cycles.

We note that, if we find $M'$ and $\varphi$ such that {\em every} set $A_u^{M',\varphi}$ contains at most one edge, then we are done by giving colour $\Delta+\mu$ to every such edge and to every edge of $M'\setminus M$. (This is the strategy we described informally before the proof.)

More importantly, we are also done if we find $M'$ and $\varphi$ such that $A_i^{M',\varphi}$ induces a subgraph that is disconnected from that of $A_j^{M',\varphi}$ for every $i\ne j\in[\Delta+\mu-1]$ as well as from $M_{\Delta+\mu}$. This is true because then we can give colour $\Delta+\mu$ to every edge in $A_i^{M',\varphi}\setminus M'$ and colour $i$ to every edge in $A_i^{M',\varphi}\cap M'$ for every $1 \leq i \leq \Delta+\mu-1$.

We now fix a choice of $M' \supseteq M$ and $\varphi : E(G\setminus M') \to [\Delta+\mu-1]$ such that $|M'|=\alpha$.
Moreover, we make our choice so that it minimises the number $\beta$ of endpoints $u$ of edges in $M$ for which $|A_u^{M',\varphi}| > 1$,
and subject to that it minimises the number $\gamma$ of edges $e$ in $M$ with endpoints $u$ and $v$ for which
there is an even index $t$ at which either of the path vertices $w^u_t$ or $w^v_t$ is at distance at least $3$ from $e$.
 
The rest of the proof is devoted to showing that under this choice any two subgraphs induced by $A_i^{M',\varphi}$, for $1\leq i \leq \Delta+\mu-1$, share no vertex and are also vertex disjoint from  $M_{\Delta+\mu}$.

\begin{claim}
For any edge in $M$ with endpoints $u$ and $v$, either $|A_u^{M',\varphi}| \le 1$ or $|A_v^{M',\varphi}| \le 1$.
\end{claim}

Suppose otherwise. 
We now construct a maximal Vizing fan pivoting on $w^u_2$, where our aim is to colour $e^u_1$ with a colour from $[\Delta+\mu-1]$ and adjust $\varphi$ so that it becomes a proper $(\Delta+\mu-1)$-edge-colouring of $G\setminus (M' \setminus \{e^u_1\})$.
If this succeeds, then there is a contradiction with the minimality of $\alpha$.

We follow a standard convention in edge-colouring by writing $\varphi(z)$ for the set of colours appearing on the edges incident to $z$ and $\overline{\varphi}(a)$ for $[\Delta+\mu-1]\setminus \varphi(z)$ where $z$ is any vertex of the graph.

Recall~\cite{SSTF12} that a {\em multi-fan} at $z$ with respect to the edge $e$ and $\varphi$ is a sequence $F = (f_1,x_1,\dots,f_p,x_p)$, $p\ge1$, where $f_1,\dots,f_p$ are distinct edges, $f_1=e$, and $f_k$ has endpoints $z$ and $x_k$ for all $k\in[p]$.
Furthermore, for every edge $f_k$, $k>1$, there is a vertex $x_\ell$, $\ell\in[k-1]$, such that $\varphi(f_k)\in \overline{\varphi}(x_\ell)$.

We have already argued above that there is no proper $(\Delta+\mu-1)$-edge-colouring of $G\setminus (M' \setminus \{e^u_1\})$. 
Let us then choose $F = (f_1,x_1,\dots,f_p,x_p)$ to be a maximal multi-fan at $w^u_2$ with respect to $e^u_1$ and $\varphi$. By Vizing's fan equation (cf.~\cite[Theorem~2.1(d)]{SSTF12}), we have that $p\ge 2$ and moreover $x_p$ cannot be incident to an edge of $M'$ as otherwise the fan would not be maximal.

Now we can ``shift'' colours along the multi-fan so that $f_1=e^u_1$ receives a colour from $[\Delta+\mu-1]$ and instead add $f_p$ to $M'$.
Note that, for any $u'\ne u$ with $|A_{u'}^{M',\phi}|\leq 1$, adding the edge $f_p$ to $M'\setminus{\{f_1\}}$ does not affect the size of $A_{u'}^{M',\phi}$ due to the distance requirement between precoloured edges. 
Thus this new choice of $M'$ and $\varphi$ (still has $|M'|=\alpha$ and) contradicts the minimality of $\beta$.
This completes the proof of the claim.

\begin{claim}
$\gamma=0$.
\end{claim}

Suppose for a contradiction that there is an edge $e$ in $M$ with endpoints $u$ and $v$ for which there is an even index $t$ at which either of the path vertices $w^u_t$ or $w^v_t$ is at distance at least $3$ from $e$. Suppose $t$ is the smallest such index. By the first claim we may assume without loss of generality that it is only the case for $w^u_t$.  Note that $w^u_t$ has distance $3$ or $4$ from $e$, and $w^u_s$ is at distance at most $2$ from $e$ for each $s<t-1$.

There is no proper $(\Delta+\mu-1)$-edge-colouring of $G\setminus (M' \setminus \{e^u_{t-1}\})$, or else there would be a contradiction with the choice of $\alpha$.
We now choose $F = (f_1,x_1,\dots,f_p,x_p)$ to be a maximal multi-fan at $w^u_t$ with respect to $e^u_{t-1}$ and $\varphi$. Again, by Vizing's fan equation, we have that $p\ge 2$ and that $x_p$ cannot be incident to an edge of $M'$. Again, we ``shift'' colours along the multi-fan so that $f_1=e^u_{t-1}$ receives a colour from $[\Delta+\mu-1]$ and instead add $f_p$ to $M'$.

Note that $x_p$ is at distance between $2$ and $5$ from $e$. It thus follows from the first claim that this new choice of $M'$ and $\varphi$ does not append $f_pw^u_te^u_t\dots$ to the path $P_v^{M',\varphi}$. So under this new choice, there is no even index $t$ at which either of the path vertices $w^u_t$ or $w^v_t$ is at distance at least $3$ from $e$. On the other hand, we could have appended $f_p$ to another path $P_w^{M',\varphi}$, but we are then guaranteed by the distance condition on $M$ that in the old choice there was already an even index $t$ for which $w^w_t$ has at distance at least $3$ from its corresponding edge in $M$. So our new choice contradicts the minimality of $\gamma$.

This completes the proof of the claim.

\medskip

This claim implies that for all $i\in[\Delta+\mu-1]$ all edges of $A_i^{M',\varphi}$ are within distance $3$ of an edge of $M_i$.
So by the distance condition on $M$ the subgraphs induced by $A_i^{M',\varphi}$, $1\leq i\leq \Delta+\mu-1$, are disconnected from one another and from $M_{\Delta+\mu}$, so this completes the proof.
\end{proof}

From our  proof we see that we could slightly relax the condition in Theorem~\ref{thm:main} on the precoloured matching. Indeed, we could demand that $M$ is a disjoint union of matchings $M'$ and $M''$, where $M'$ has minimum distance at least $9$ and is arbitrarily precoloured from $[\Delta+\mu-1]$, $M''$ is precoloured $\Delta+\mu-1$, and the minimum distance between an edge in $M'$ and an edge in $M''$ is at least $4$.

%We are also able to drop the distance requirement to $5$ if the graph contains no cycle of length $5$. 
The proof of Theorem~\ref{thm:c5free} is conceptually the same as that of Theorem~\ref{thm:main}, but simpler.

\begin{proof}[Proof of Theorem~\ref{thm:c5free}]
Just as before, let $\Phi : M \to \mathcal{K}$ be the precolouring on $M$. Let $\alpha$ be the cardinality of a smallest matching $M' \supseteq M$ such that there exists a proper $(\Delta+\mu-1)$-edge-colouring of $G\setminus M'$. Theorem~\ref{thm:BeFo} certifies that $\alpha$ is well defined.
 
For any matching $M' \supseteq M$ and any proper edge-colouring $\varphi : E(G\setminus M') \to [\Delta+\mu-1]$
we say that an edge $e\in M$ is {\em bad} if there exist $e_1$ and $e_2$ such that $\varphi (e_{1})=\Phi(e)$, $e_{2}\in M'\setminus M$ and $e_1$ is adjacent to both $e$ and $e_2$. If there are no bad edges, then we may extend  $\Phi$ to a proper edge-colouring of $G$ by colouring any edge $e\notin M'$ with $\varphi(e)$ and any edge $e\in M'\setminus M$ with $\Delta + \mu$. 
We now fix a choice of $M'$ and $\varphi$ with $|M'|=\alpha$ and, subject to this, having the least number $\beta$ of bad edges. 
The rest of the proof is devoted to showing $\beta=0$. 

Suppose $e\in M$ is a bad edge and let $e_1$ and $e_2$ be edges certifying its badness as defined above. 
We have that $e$, $e_1$, $e_2$ form a path of length $3$. Calling $w_2$ the endpoint of this path that is incident with $e_2$, let $F = (f_1,x_1,\dots,f_p,x_p)$ be a maximal multi-fan at $w_2$ with respect to $e_2$
and $\varphi$. As in the previous proof, by Vizing's fan equation $p\ge 2$ and $x_p$ is not incident with an edge of $M'$. By ``shifting'' the colours along $F$, we can colour $f_1=e_2$ from $[\Delta+\mu-1]$ and add $f_p$ to $M'$. 

Under this new choice of $\varphi$ and $M'$ (which still has $|M'|=\alpha$), any new bad edge would have to be within distance $1$ of $f_p$ and thus within distance $4$ of $e$, contradicting the distance requirement on $M$. 
Due to the shift, $e$ can no longer be certified bad with the help of $e_1$ since $e_1$ is no longer incident to an edge of $M'$.
However, by the choice of $\beta$, it must be that $e$ has remained bad with respect to the new choice of $\varphi$ and $M'$.
So there exist $e_1'$ and $e'_2$ certifying that $e$ is bad such that $\{e_1,e_2\}\cap \{e_1',e_2'\} =\emptyset$ and the union of endpoints of $e_1$ and $e_2$ is disjoint from that of $e_1'$ and $e_2'$.
We have furthermore that $e_2' \ne f_p$ or else 
$e$, $e_1$, $e_2$, $f_p$, $e_1'$ would form a cycle of length $5$. Clearly $e_2'$ and $f_p$ are not incident as both belong to $M'$.

We can perform another pivot like before but instead at the end of the path $ee_1'e_2'$. Calling $w_2'$ the endpoint of this path that is incident with $e_2'$, a maximal multi-fan $F' = (f_1',x_1',\dots,f_{p'}',x_{p'}')$ at $w_2'$ with respect to $e_2'$ and $\varphi$ must have $p'\ge 2$ and must end at a vertex $x_{p'}'\notin M'$.
In particular, $x_{p'}'$ and $f_p$ are not incident.
Moreover $x_{p'}$ is not the common endpoint of $e_1$ and $e_2$ or else there would be a cycle of length $5$.
Again we shift the colours along $F'$ so as to colour $f_1'=e_2'$ from $[\Delta+\mu-1]$ and add $f_{p'}'$ to $M'$.
 
Arguing in the same way as before, under this second new choice of $\varphi$ and $M'$ (which also still has $|M'|=\alpha$), there is no new bad edge.
Note that we have now modified $\varphi$ and $M'$ so that neither $e_1$ nor $e_1'$ may help to certify that $e$ is bad. 
Thus $e$ is no longer bad since in any proper partial edge-colouring there are at most two edges incident to $e$ coloured $\Phi(e)$.
This is a contradiction to the choice of $\beta$.

We may therefore conclude that $\beta=0$ and this completes the proof.
\end{proof}

\bibliographystyle{abbrv}
\bibliography{edgeprecolour2}

%%%%%%%%%%%%%%%%%%%%%%%%%%%%%%%%%%%%%%%%%%%%%%%%%%%%%%%%%%%%%%%%%%%%%%%%

\end{document}